\documentclass[12pt,a4paper]{article}

\usepackage[margin=29mm]{geometry}
\usepackage{amsmath,amssymb,amsthm,mathtools}
\usepackage{enumitem}
\usepackage{microtype}
\usepackage{xcolor}
\usepackage{hyperref}
\usepackage[nameinlink,capitalize,noabbrev]{cleveref}
\usepackage{tikz}
\usetikzlibrary{arrows.meta,decorations.pathreplacing}

\hypersetup{
	colorlinks=true,
	linkcolor=blue,
	filecolor=blue,
	urlcolor=blue,
	citecolor=cyan,
}

\allowdisplaybreaks
\numberwithin{equation}{section}
\setlist[enumerate]{leftmargin=2.2em,itemsep=2pt,topsep=3pt}

\newtheorem{theorem}{Theorem}[section]
\newtheorem{lemma}[theorem]{Lemma}

\newtheorem{corollary}[theorem]{Corollary}
\newtheorem{claim}[theorem]{Claim}
\theoremstyle{definition}

\theoremstyle{remark}

\title{Near-optimal Tur\'an densities of $r$-graphs on $r+1$ vertices}
\author{Jiabao Yang\thanks{School of Mathematics, Nanjing University, Nanjing 210093, China (Email: jbyang1215@nju.edu.cn)}~,
Xiutao Zhu\thanks{Corresponding author. School of Mathematics, Nanjing University of Aeronautics and Astronautics, Nanjing 211106, China (Email: zhuxiutao@nuaa.edu.cn)}
}
\date{}

\begin{document}
\maketitle
\vspace{-1.2em}

\begin{abstract}
Let $\pi(H)$ be the Tur\'an density of an r-uniform hypergraph $H$ and let $H_k^r$ denote the $r$-uniform hypergraph on $r+1$ vertices with exactly $k$ edges, where $1\le k\le r+1$.  
Sidorenko~(JCT-B, 2024) proved that
$\pi(H_3^r)\ge (1.7215-o(1))r^{-2}$ as $r\to\infty$ and 
$\pi(H_k^r)\ge (C_k+o(1))r^{-(1+1/(k-2))}$ for fixed $k$ as $r\to\infty$.  
Clemen~later improved the first bound to $\pi(H_3^r)\ge cr^{-2}\sqrt{\log r}$ for some constant $c>0$.

In this article, we prove the following results.
\begin{itemize}
\item For any fixed $\varepsilon>0$, there is a constant $c_\varepsilon>0$ such that
$$\pi(H_3^r)\ge \frac{c_\varepsilon}{r(\log r)^{2+\varepsilon}}.$$
Together with the known upper bound $\pi(H_3^r)\le1/r$, this implies
$\pi(H_3^r)=r^{-1+o(1)}$.

\item For every $3\le k\le r+1$, let
$s=\min\{k-2,r-k+2\}$.  Then
\begin{equation*}
 0\le \frac{k-2}{r}-\pi(H_k^r)
 \le \frac{128}{r}\left(\sqrt{s\log\frac{er}{s}}+\log\frac{er}{s}\right).
\end{equation*}
This estimate yields several asymptotically sharp results for $\pi(H_k^r)$.   For
example, $\pi(H_k^r)=(1+o(1))(k-2)/r$
when $\log(er/(k))=o(k)$.

\end{itemize}
\end{abstract}

\medskip
\noindent\textbf{Keywords.} Tur\'an density; Tur\'an-type problem; hypergraph; circular construction

\medskip
\noindent\textbf{Remark:} After uploading this paper, the two authors together with Olep Pikhurko investigated the more general problems for hypergraph augmentations and determined the Tur\'an density of each $r$-graph $H_k^r$ within a constant factor. See the newest version at arXiv:2610.11757v2.

\section{Introduction and main results}

Let $H$ be an $r$-uniform hypergraph, or simply an $r$-graph.
An $r$-graph $G$ is called {\em $H$-free} if $G$ does not contain a copy of the $r$-graph $H$ as a subgraph.
The {\em Tur\'{a}n number} of the $r$-graph $H$, denoted by $\mathrm{ex}(n,H)$, is the maximum number of edges in an $H$-free $r$-graph on $n$ vertices.

Over the past century, Tur\'an problems have awalys been a central topic in extremal combinatorics. 
Compared to the exact value of $\mathrm{ex}(n,H)$,  usually we are more interested in the asymptotic behavior of $\mathrm{ex}(n,H)$, which is called the Tur\'an density of $H$, denoted by
$\pi(H)=\lim_{n\to\infty} \mathrm{ex}(n,H)/\binom {n}{r}.$

For $2$-graphs, Tur\'an densities are well understood. The classical Erd\H{o}s-Stone-Simonovits Theorem \cite{Er, stone} tells us $\pi(H)=1-\frac{1}{\chi(H)-1}$. However,  determining them for $r$-graphs with $r\ge 3$ is difficult in general.
the Tur\'an densities $\pi(K_4^{(3)-})$ and $\pi(K_4^{(3)})$  remain unknown, where $K_4^{(3)}$ is the complete $3$-graph on four vertices and $K_4^{(3)-}$ is obtained from it by deleting one edge.
Determining $\pi(K_4^{(3)})$ remains one of the main open problems in extremal combinatorics.
For further background and related results, we refer the reader to~\cite{deCaen1994,Furedi1991,Keevash2011,Sidorenko1995}.

In this paper, we mainly focus on a family of Tur\'an problem. For $3\le k\le r+1$, let $H_k^r$ be the $r$-graph on $r+1$ vertices with $k$ edges. Clearly, $H_4^3=K_4^{(3)}$, $H_3^3=K_4^{(3)-}$. And we will study the Tur\'an density $\pi(H_k^r)$ in this paper. Actually, this problem  is also a special case of Brown-Erd\H{o}s-S\'os problem. Brown, Erd\H{o}s and S\'os \cite{BES} initiated the study of $f^{(r)}(n,v,e)$ which is the smallest $f$ such that every $r$-graph with $n$ vertices and $f$ edges has a subset of $v$ vertices with at least $e$ edges. The literature on this topic is extensive, see \cite{Alon,Postle,Glock,Glock2,Wang}. If we let
$$\mathcal{H}_e^v=\{H:\text{$H$ is a $r$-graph with $v$ vertices and $e$ edges}\}$$
then $\mathrm{ex}(n,\mathcal{H}_e^v)=f^{(r)}(n,v,e)-1$. Since every $r$-subset of an $(r+1)$-set omits exactly one vertex,  $H_k^r$ is unique up to isomorphism and $\mathcal{H}_k^{r+1}=\{H_k^r\}$. So studying the Tur\'an problem $\mathrm{ex}(n,H_k^r)$ and $\pi(H_k^r)$ is equivalent to Brown-Erd\H{o}s-S\'os problem $f^{(r)}(n,r+1,k)$.

There are many rich results about the problem $\pi(H_k^r)$. Since $H_3^2$ is a triangle, Mantel's theorem~\cite{Mantel1907} shows that $\pi(H_3^2)=1/2$. Frankl and F\"uredi~\cite{FranklFuredi1984} conjectured that $\pi(H_3^3)=2/7$, and Gunderson and Semeraro~\cite{GundersonSemeraro2017} proved that $\pi(H_3^4)=1/4$.
Gunderson and Semeraro~\cite{GundersonSemeraro2025} also proved $\pi(H_3^6)\ge 9/64$, $\pi(H_3^7)\ge 35/2^{11}$, and $\pi(H_3^8)\ge 315/2^{14}$.
By considering the link graph of the vertex with maximum degree, one can easily get\footnote{For the detailed proof, see the introduction in \cite{Sidorenko2024} } 
\begin{equation}\label{eq:monotonicity}
\pi(H_k^{r-1})\ge \pi(H_k^r).
\end{equation}
Using this inequality, we can also get $\pi(H_3^5)\ge\pi(H_3^6)\ge 9/64$. For the general $H_3^r$, earlier bounds showed $2^{1-r}\le \pi(H_3^r)\le 1/r$,
see Frankl and F\"uredi~\cite{FranklFuredi1984},  Markstr\"om and Thomassen~\cite{MarkstromThomassen2021}, respectively. Recently,
Sidorenko improved the general lower bound to $r^{-2}$.
\begin{theorem}\label{end:lem:positive}(Sidorenko~\cite{Sidorenko2024})
We have $\pi(H_3^r)\ge 1/r^{2}$ for every $r$ and
$\pi(H_3^r)\ge (1.7215-o(1))/r^{2}$ as $r\to\infty$.   
\end{theorem}
Later, 
Clemen~\cite{Clemen2026} improved the  lower bound to $\pi(H_3^r)\ge c r^{-2}\sqrt{\log r}$ for some absolute constant $c>0$, where $\log$ denotes the natural logarithm.
In this paper, we first prove the following result.

\begin{theorem}\label{end:thm:main}
For every fixed $\varepsilon>0$, there is a constant $c_\varepsilon>0$ such that
$$\pi(H_3^r)\ge \frac{c_\varepsilon}{r(\log r)^{2+\varepsilon}}$$
for all sufficiently large $r$.  
Hence $\pi(H_3^r)=r^{-1+o(1)}$.
\end{theorem}

For $k\ge4$, the general upper bound $\pi(H_k^r)\le (k-2)/r$ is known \cite{MarkstromThomassen2021}.
But , fewer results  for the lower bounds are known. For example, $\pi(H_5^4)\ge 11/16$ was proved in~\cite{deCaenKreherWiseman1988,Giraud1990}, $\pi(H_4^3)\ge 5/9$ and $\pi(H_{r+1}^r)\ge 1-\left(1/2+o(1)\right)\frac{\log r}{r}$ as $r\to\infty$; see~\cite{Sidorenko1995,Sidorenko1997}, and  $\pi(H_4^4)\ge 3/7$  was given in \cite{Sidorenko2024}. Later, Sidorenko  gave a general lower bound of $\pi(H_k^r)$ using the construction  of random points uniformly distributed on the circle: choose $r$ independent points uniformly at random on the circle of circumference one.
For $1\le t\le r$, let $\xi_{r,t}$ be the length of the shortest arc that contains at least $t$ of the points.
For each fixed $m\ge1$, Cressie~\cite{Cressie1977Gap} proved that
$$ \mathbb{E}(\xi_{r,m+1})=\left((m!)^{1/m}\Gamma\left(1+\frac1m\right)+o(1)\right)r^{-1-1/m},$$
where $\Gamma$ denotes the gamma function. Based on this construction, Sidorenko proved the following lower bounds.

\begin{theorem}[Sidorenko~\cite{Sidorenko2024}]\label{thm:sidorenko}
For $3\le k\le r+1$, $\pi(H_k^r)\ge\mathbb{E}(\xi_{r,k-1})$. Moreover, for every fixed $k$, 
$\mathbb{E}(\xi_{r,k-1})=(C_k+o(1))r^{-1-\frac{1}{k-2}}$ for some constant $C_k$ as $r\to \infty$.
\end{theorem}

One may note, this lower bound is still very far from the previous upper bound  $(k-2)/r$.
However, when combining this construction and R\'enyi exponential representation, we still get better results when $k$ grows with $r$. That is 

\begin{theorem}\label{thm:main}
For all integers $r\ge2$ and $3\le k\le r+1$, 
\begin{equation*}
 0\le \frac{k-2}{r}-\pi(H_k^r)\le \frac{128}{r}\left(\sqrt{s \log\frac{er}{s}}+\log\frac{er}{s}\right),
\end{equation*}
where $s=\min\{k-2,r-k+2\}$.
\end{theorem}

In some sense, the error in the right hand would exceed $(k-2)/r$. For example when $k=3$, the error is much larger than $(k-2)/r$. However, when $k=k(r)$ grows with $r$, sometimes we can get much better lower bound using our result.   As an application, we obtain the following three different asymptotic forms.

\begin{corollary}\label{cor:applications}
Let $k=k(r)$ be an integer with $3\le k\le r+1$.
\begin{enumerate}[label=\textup{(\roman*)}]
\item If $k/\log r\to c\in(0,\infty)$, then
\begin{equation*}
1-128\left(c^{-1/2}+c^{-1}\right)+o(1)
\le \frac{r\pi(H_k^r)}{k-2}\le1.
\end{equation*}
\item If $\log(er/k)=o(k)$, then
\begin{equation*}
\pi(H_k^r)=(1+o(1))\frac{k-2}{r}.
\end{equation*}
\item If $k/r\to\alpha\in(0,1)$ and $\beta=\min\{\alpha,1-\alpha\}$, then
\begin{equation*}
0\le\frac{k-2}{r}-\pi(H_k^r)
\le\frac{128\sqrt{\beta\log(e/\beta)}+o(1)}{\sqrt r}.
\end{equation*}
\end{enumerate}
\end{corollary}

\medskip
{\bf\noindent Notation.} Throughout this paper, we write $[n]=\{1,\ldots,n\}$. For integers $a\ge h\ge0$, we write $(a)_h=a(a-1)\cdots(a-h+1)$, with $(a)_0=1$.  
We write $a_r=o(b_r)$ when $a_r/b_r\to0$ and $a_r=O(b_r)$ when $|a_r/b_r|$ remains bounded.

\subsection*{Organization and proof outline}

We now describe the organization of the paper together with a proof outline.

In Section \ref{sec:2}, we prove Theorem \ref{end:thm:main}.  The proof of Theorem \ref{end:thm:main} uses the following recursive estimate: $\pi(H_3^r)\ge \frac{1}{360r\log r}\,\pi\bigl(H_3^{\lceil 129\log r\rceil}\bigr).$
We first use a dense $H_3^t$-free family and independent random permutations to construct a local colouring with few colours.  
We then divide a large vertex set into equal parts and consider the $r$-sets that meet every part in at most $\lceil 128\log r\rceil$ vertices.  
Each such $r$-set is assigned a label with two residue coordinates.  
If three edges of a possible copy of $H_3^r$ have the same label, the first coordinate forces the three deleted vertices to lie in the same part, 
while the second coordinate contradicts the property of the local colouring.  
The pair-covering blow-up lemma changes this finite construction into the recursive lower bound in Theorem \ref{end:thm:recursion}.  
A product estimate then controls its iteration and proves Theorem \ref{end:thm:main}.

In Section \ref{sec:3}, we prove Theorem \ref{thm:main} and Corollary \ref{cor:applications}.  
Sidorenko's circular construction and the R\'enyi exponential representation reduce the required lower bound to estimating the minimum of cyclic sums of independent exponential random variables.  
After centering these variables, we divide the cyclic sums into blocks and apply a maximal inequality for ordinary partial sums.  
This gives a uniform estimate for the minimum when $m\le r/2$.  When $m>r/2$, a complement identity reduces the problem to cyclic sums of length $r-m$.  Combining the resulting lower bound with the known upper bound proves the finite estimate in Theorem \ref{thm:main}. We later use Theorem \ref{thm:main} to prove Corollary \ref{cor:applications}.

\section{Results for the graph \texorpdfstring{$H_3^r$}{H3r}}\label{sec:2}
In this section, we first establish the following recursive estimate and then use it to prove Theorem \ref{end:thm:main}.

\begin{theorem}\label{end:thm:recursion}
There is an absolute constant $r_0$ such that, for every $r\ge r_0$,
$$\pi(H_3^r)\ge \frac{1}{360r\log r}\,\pi\bigl(H_3^{\lceil 129\log r\rceil}\bigr).$$
\end{theorem}

To prove our results, we need some preliminary results. The following lemma is folklore. 

\begin{lemma}\label{end:lem:finite-density}
Let $F$ be an $r$-graph, and for $n\ge r$ define $d_F(n)=\mathrm{ex}(n,F)/\binom nr$.
Then $d_F(n)$ is nonincreasing in $n$.
Hence $d_F(n)\ge \pi(F)$ for every $n\ge r$.
\end{lemma}

An $r$-graph is \emph{pair-covering} if every pair of its vertices is contained in an edge.  
The $r$-graph $H_3^r$ is pair-covering.  
Indeed, we write its edges as $S\setminus\{x\}$, $S\setminus\{y\}$, and $S\setminus\{z\}$, where $x,y,z$ are distinct.  
Any two vertices of $S$ lie in the edge obtained by deleting one of $x,y,z$ outside that pair. The following result of Sidorenko~\cite{Sidorenko2024} relates the Tur\'an density of a pair-covering $r$-graph to its finite Tur\'an numbers.
\begin{lemma}[Sidorenko~\cite{Sidorenko2024}]\label{end:lem:blowup}
For any pair-covering $r$-graph $F$, 
$\pi(F)\ge r!n^{-r}\mathrm{ex}(n,F)$ holds for all $n$.
\end{lemma}

The following standard form of the hypergeometric upper-tail estimate follows from Chv\'atal~\cite{Chvatal1979}.

\begin{lemma}\label{end:lem:hypergeom}
Let $[n]$ be the $n$-element set and $S\subseteq [n]$.  We randomly choose an $r$-element subset $B$ from $[n]$ and let $X$ be the random variable $X=|B\cap S|$.
For every integer $h\ge1$, we have
$$\mathbb{P}(X\ge h)\le \left(\frac{e \,\mathbb{E}(X)}{h}\right)^h.$$
Consequently, if $h\ge4 \,\mathbb{E}(X)$, then $\mathbb{P}(X\ge h)\le(e/4)^h$.
\end{lemma}

The following lemma covers all $r$-subsets of $[n]$ by permuted copies of one dense $H_3^r$-free 
family and colours each $r$-set by the first copy that contains it.

\begin{lemma}\label{end:lem:local-colour}
Let $n\ge r+1$ and $r\ge2$.  There is an edge colouring $\varphi_r$ of
$\binom{[n]}{r}$ using at most
$$ \left\lceil\frac{1+\log\binom nr}{\pi(H_3^r)}\right\rceil$$
colours such that there is no monochromatic copy of $H_3^r$.

\end{lemma}

\begin{proof}
By Lemma \ref{end:lem:finite-density}, there is an $H_3^r$-free graph  $\mathcal I\subseteq\binom{[n]}r$ such that $|\mathcal I|/\binom nr\ge\pi(H_3^r)$.  Let
$t=\left\lceil(1+\log\binom nr)/\pi(H_3^r)\right\rceil$.
Choose permutations $\pi_1,\ldots,\pi_t$ of $[n]$ independently and uniformly at random.  Thus, for each $j\in[t]$, every permutation of $[n]$ is chosen with probability $1/n!$.  For $I\subseteq[n]$, let $\pi_j(I)=\{\pi_j(x):x\in I\}$, and let $\mathcal I_j=\{\pi_j(I):I\in\mathcal I\}$.  Since a permutation only changes the labels of the vertices, it preserves copies of $H_3^r$.  Hence every $\mathcal I_j$ is $H_3^r$-free.

Fix $A\in\binom{[n]}r$.  For each fixed $I\in\binom{[n]}r$, exactly $r!(n-r)!$ permutations of $[n]$ map $I$ onto $A$. 
Therefore $\mathbb{P}(\pi_j(I)=A)=r!(n-r)!/n!=1/\binom nr$.
Moreover, for distinct $I,I'\in\mathcal I$, the events $\pi_j(I)=A$ and $\pi_j(I')=A$ are disjoint, since $\pi_j$ is a bijection.  
It follows that $\mathbb{P}(A\in\mathcal I_j)=|\mathcal I|/\binom nr$.

Since $\pi_1,\ldots,\pi_t$ are independent, the events $A\notin\mathcal I_1,\ldots,A\notin\mathcal I_t$ are independent.  Hence
$$ \mathbb{P}(A\notin\mathcal I_1\cup\cdots\cup\mathcal I_t)
 =\left(1-\frac{|\mathcal I|}{\binom nr}\right)^t
 \le \exp\left(-\frac{t|\mathcal I|}{\binom nr}\right)
 \le e^{-\pi(H_3^r)t}.$$
The first inequality follows from $1-u\le e^{-u}$ for $0\le u\le1$, and the second follows from $|\mathcal I|/\binom nr\ge\pi(H_3^r)$.

For each $A\in\binom{[n]}r$, let $X_A$ be the indicator of the event that $A$ does not belong to $\mathcal I_1\cup\cdots\cup\mathcal I_t$, and let $X=\sum_{A\in\binom{[n]}r}X_A$.  Thus, $X$ is the number of uncovered $r$-sets.  By linearity of expectation,
$$ \mathbb{E}(X)
 =\sum_{A\in\binom{[n]}r}\mathbb{P}(A\notin\mathcal I_1\cup\cdots\cup\mathcal I_t)
 \le\binom nr e^{-\pi(H_3^r)t}.$$
By the choice of $t$, we have $\pi(H_3^r)t\ge1+\log\binom nr$, and hence
$$ \mathbb{E}(X)\le\binom nr\exp\left(-1-\log\binom nr\right)=e^{-1}<1.$$
Since $X$ is a nonnegative integer-valued random variable, there is a choice of $\pi_1,\ldots,\pi_t$ for which $X=0$.  
Indeed, if every choice left at least one $r$-set uncovered, then $X\ge1$ for every choice, 
which would imply $\mathbb{E}(X)\ge1$.  
We may therefore fix permutations satisfying
$\mathcal I_1\cup\cdots\cup\mathcal I_t=\binom{[n]}r$.

For each $A\in\binom{[n]}r$, it would belong to many $\mathcal I_j$'s. We only keep it in one $\mathcal I_j$. That is let $\mathcal I_j'=\mathcal I_j\setminus (\mathcal I_1\cup \cdots \cup \mathcal I_{j-1}) $. Then  $\mathcal I_1',\ldots,\mathcal I_t'$ is an edge coloring of $\binom{[n]}{r}$ using at most $t$ colors. Each $\mathcal{I}_j'$ is a subgraph of $\mathcal{I}_j$, so it is still $H_3^r$-free. 

The proof is complete.
\end{proof}

\medskip
We now prove Theorem \ref{end:thm:recursion}.

\begin{proof}[\bfseries{Proof of Theorem \ref{end:thm:recursion}}]
Fix a sufficiently large integer $r$.  We first define four parameters:
$$ q=\left\lfloor\frac{r}{16\log r}\right\rfloor,
 \quad L=\lceil128\log r\rceil,
 \quad M=\left\lceil\frac{r^2}{q}\right\rceil,
 \quad N=qM.$$
We will later define the common number $P$ of local colours.  The first four parameters satisfy
\begin{equation}\label{end:eq:param-estimates}
 \frac{r}{32\log r}\le q\le\frac{r}{16\log r},
 \quad 4r\le qL\le10r,
 \quad 16r\log r\le M\le48r\log r.
\end{equation}
For sufficiently large $r$, one has $r/(16\log r)\ge2$, and hence $\lfloor r/(16\log r)\rfloor\ge r/(32\log r)$.  This proves the two bounds on $q$.  Since $128\log r\le L\le128\log r+1$, the lower bounds on $q$ and $L$ give $qL\ge4r$, while
$qL\le r(128\log r+1)/(16\log r)\le10r$ for sufficiently large $r$.  Finally, $q\le r/(16\log r)$ gives $M\ge r^2/q\ge16r\log r$, and $q\ge r/(32\log r)$ gives
$M\le r^2/q+1\le32r\log r+1\le48r\log r$.  
Thus \eqref{end:eq:param-estimates} holds.

Partition an $N$-element set $V$ into $q$ pairwise disjoint sets $V_0,\ldots,V_{q-1}$, each of size $M$.  Fix $i\in\{0,\ldots,q-1\}$ and $t$ with $2\le t\le L$.  
We apply Lemma \ref{end:lem:local-colour} to the $t$-subsets of $V_i$.  
Since $t\le L$, repeated use of \eqref{eq:monotonicity}, together with Theorem  \ref{end:lem:positive}, yields $\pi(H_3^t)\ge\pi(H_3^L)>0$.  
Thus Lemma \ref{end:lem:local-colour} provides an edge-colouring of $\binom{V_i}t$ in which every $(t+1)$-subset of $V_i$ contains at most two $t$-subsets of any one colour, and the number of colours is at most $\left\lceil(1+\log\binom Mt)/\pi(H_3^t)\right\rceil$.

The standard estimate $\binom Mt\le(eM/t)^t$ yields $\log\binom Mt\le t\log(eM/t)\le L\log(eM)$.  Since $L\log(eM)\ge1$ for all sufficiently large $r$, we have $1+\log\binom Mt\le2L\log(eM)$.  Moreover, $\pi(H_3^t)\ge\pi(H_3^L)$.  It follows that each local colouring uses at most
$$ \left\lceil\frac{1+\log\binom Mt}{\pi(H_3^t)}\right\rceil
 \le \left\lceil\frac{2L\log(eM)}{\pi(H_3^L)}\right\rceil
 \le P,
 \quad \text{where} \quad
 P:=\left\lceil\frac{3L\log(eM)}{\pi(H_3^L)}\right\rceil.$$
After relabelling the colours if necessary, we may regard every local colouring as taking values in $\{0,1,\ldots,P-1\}$, with some colour labels left unused when fewer than $P$ colours are required.  We denote the colouring corresponding to $(i,t)$ by
$\varphi_{i,t}:\binom{V_i}t\to\{0,1,\ldots,P-1\}$.
For $t=0$ and $t=1$, let $\varphi_{i,t}$ be the constant zero map.  The required local property holds automatically in these two cases, since a $(t+1)$-set has only $t+1\le2$ subsets of size $t$ and therefore cannot contain three such subsets of the same colour.

Call an $r$-set $B\subseteq V$ \emph{good} if $|B\cap V_i|\le L$ for every $i\in\{0,\ldots,q-1\}$, 
and let $\mathcal G$ be the family of all good $r$-sets.  
We first estimate the size of $\mathcal G$.  

\begin{claim}\label{claim:eq:good-count}
We have 
$ |\mathcal G|=(1-o(1))\binom Nr.$
\end{claim}

\begin{proof}
Choose $B$ uniformly at random from $\binom Vr$ and fix $i\in\{0,\ldots,q-1\}$.  The random variable $X_i:=|B\cap V_i|$ has the hypergeometric distribution obtained by choosing $r$ elements without replacement from a set of size $N$, of which the $M$ elements in $V_i$ are distinguished.  Since $N=qM$, it means
$\mathbb{E}( X_i)=r|V_i|/|V|=rM/N=r/q$.
By \eqref{end:eq:param-estimates}, we have $16\log r\le r/q\le32\log r$, while $L\ge128\log r$.  
In particular, $L\ge4\mathbb{E}(X_i)$.

Applying Lemma \ref{end:lem:hypergeom} with $h=L$, we obtain
$$ \mathbb{P}(|B\cap V_i|>L)
 \le\mathbb{P}(X_i\ge L)
 \le\left(\frac{e\mathbb{E}(X_i)}{L}\right)^L
 \le(e/4)^L
 \le r^{-20}.$$
For the last inequality, since $L\ge128\log r$,
$(e/4)^L\le\exp(-128\log r\log(4/e))=r^{-128\log(4/e)}$,
and $128\log(4/e)>20$.

The event $B\notin\mathcal G$ occurs only if $|B\cap V_i|>L$ for at least one $i\in\{0,\ldots,q-1\}$.  Therefore, by the union bound and the estimate $q\le r$ from \eqref{end:eq:param-estimates},
$\mathbb{P}(B\notin\mathcal G)\le\sum_{i=0}^{q-1}\mathbb{P}(|B\cap V_i|>L)\le qr^{-20}\le r^{-19}=o(1)$.
Since $B$ is chosen uniformly from $\binom Vr$, we have
$\mathbb{P}(B\in\mathcal G)=|\mathcal G|/\binom Nr$.
Accordingly,
\begin{equation*}
 |\mathcal G|=(1-o(1))\binom Nr,
\end{equation*}
as required.
\end{proof}

For each $B\in\mathcal G$, assign the ordered pair of residues
\begin{equation}\label{end:eq:two-coordinate-label}
 \left(
   \sum_{i=0}^{q-1}i|B\cap V_i|\bmod q,
   \quad
   \sum_{i=0}^{q-1}\varphi_{i,|B\cap V_i|}(B\cap V_i)\bmod P
 \right).
\end{equation}
This label is well defined.  Indeed, since $B$ is good, we have $0\le |B\cap V_i|\le L$ for every $i$.  Hence $B\cap V_i$ belongs to the domain of $\varphi_{i,|B\cap V_i|}$, including the cases $|B\cap V_i|=0$ and $|B\cap V_i|=1$.  The first coordinate of \eqref{end:eq:two-coordinate-label} is a residue modulo $q$ and therefore has $q$ possible values.  Similarly, the second coordinate is a residue modulo $P$ and has $P$ possible values.  Thus there are at most $qP$ possible labels.

The labels divide $\mathcal G$ into at most $qP$ classes.  Let $\mathcal F$ be a largest such class.  By the pigeonhole principle,
$|\mathcal F|\ge|\mathcal G|/(qP)$.
Using Claim \ref{claim:eq:good-count}, we obtain
\begin{equation}\label{end:eq:class-size}
 |\mathcal F|\ge\frac{|\mathcal G|}{qP}
 =(1-o(1))\frac1{qP}\binom Nr.
\end{equation}

\begin{claim}\label{claim:f-is-h3r-free}
$\mathcal F$ is $H_3^r$-free.
\end{claim}
\begin{proof}
Suppose, for the sake of contradiction, that some $(r+1)$-set $S\subseteq V$ contains three distinct vertices $x,y,z$ such that
$S\setminus\{x\}$, $S\setminus\{y\}$, and $S\setminus\{z\}$ all belong to $\mathcal F$.  
Since $\mathcal F\subseteq\mathcal G$, these three $r$-sets are good, and hence their labels in \eqref{end:eq:two-coordinate-label} are well defined.  
Let $x\in V_a$, $y\in V_b$, and $z\in V_c$, where $a,b,c\in\{0,\ldots,q-1\}$.

For any $u\in S\cap V_j$, deleting $u$ decreases $|S\cap V_j|$ by one and leaves all other intersection sizes unchanged.  
Therefore the first coordinate of the label of $S\setminus\{u\}$ is equal to $\sum_{h=0}^{q-1}h|S\cap V_h|-j$ modulo $q$.
In particular, the first coordinates of the labels of $S\setminus\{x\}$, $S\setminus\{y\}$, and $S\setminus\{z\}$ are congruent to
$\sum_{h=0}^{q-1}h|S\cap V_h|-a$,
$\sum_{h=0}^{q-1}h|S\cap V_h|-b$, and
$\sum_{h=0}^{q-1}h|S\cap V_h|-c$,
respectively.  
All members of $\mathcal F$ have the same label, so these three residues are equal, which means $a\equiv b\equiv c\pmod q$.
Since $a,b,c\in\{0,\ldots,q-1\}$, it follows that $a=b=c$.  
Thus $x,y,z$ all lie in the same part, which we denote by $V_a$.

Let $A=S\cap V_a$.  Since $x,y,z\in V_a$, the intersections of the three sets $S\setminus\{x\}$, $S\setminus\{y\}$, and $S\setminus\{z\}$ with $V_a$ are
$A\setminus\{x\}$, $A\setminus\{y\}$, and $A\setminus\{z\}$,
respectively.  These are distinct sets, each of size $|A|-1$.  Since all three $r$-sets belong to $\mathcal F\subseteq\mathcal G$, they are good, and hence $|A|-1\le L$.  Moreover, $x,y,z\in A$ are distinct, so $|A|\ge3$ and therefore $|A|-1\ge2$.  Thus the local colouring
$\varphi_{a,|A|-1}:\binom{V_a}{|A|-1}\to\{0,\ldots,P-1\}$
is defined and has the required avoidance property.

For every $h\ne a$, deleting any one of $x,y,z$ does not change the intersection with $V_h$.  Hence the contributions from all parts other than $V_a$ to the second coordinate in \eqref{end:eq:two-coordinate-label} are the same for the three sets.  Since all members of $\mathcal F$ have the same label, their second coordinates are equal modulo $P$.  Cancelling the common contributions from the parts $V_h$ with $h\ne a$, we obtain
$$ \varphi_{a,|A|-1}(A\setminus\{x\})
 \equiv\varphi_{a,|A|-1}(A\setminus\{y\})
 \equiv\varphi_{a,|A|-1}(A\setminus\{z\})\pmod P.$$
Each of these three values belongs to $\{0,\ldots,P-1\}$, so the congruences imply that the three values are equal.  Consequently, the $|A|$-set $A$ contains three distinct $(|A|-1)$-subsets of the same colour under $\varphi_{a,|A|-1}$.  This contradicts the defining property of the local colouring, which states that every $|A|$-set contains at most two $(|A|-1)$-subsets of any one colour.  Therefore $\mathcal F$ contains no copy of $H_3^r$, and hence $\mathcal F$ is $H_3^r$-free.
\end{proof}

Now, applying Lemma \ref{end:lem:blowup} to the $r$-graph on $V$ with edge set $\mathcal F$ and using Claim \ref{claim:f-is-h3r-free} and inequality \eqref{end:eq:class-size}, we obtain
\begin{align}\label{end:eq:density-before-final}
  \pi(H_3^r)&\geq \frac{r!}{N^r}|\mathcal{F}|\ge (1-o(1))\frac{1}{qP}\binom{N}{r}\frac{r!}{N^r}=(1-o(1))\frac{1}{qP}\frac{(N)_r}{N^r}\ge \frac{1}{3qP}, 
\end{align}
where the last inequality holds since $\frac{(N)_r}{N^r}>\frac12$ and $1-o(1)\ge \frac{2}{3}$ when $r$ is large.

It remains to bound $qP$.  Since $0<\pi(H_3^L)\le1$ and $L\log(eM)\ge1$, we have

$$ P\le1+\frac{3L\log(eM)}{\pi(H_3^L)}
 \le\frac{4L\log(eM)}{\pi(H_3^L)}.$$
Together with $qL\le10r$, this yields
$qP\le40r\log(eM)/\pi(H_3^L)$. Also since $M\le48r\log r$ and $48e\log r\le r^2$, it implies $\log(eM)\le3\log r$ for sufficiently large $r$.  Thus
$qP\le120r\log r/\pi(H_3^L)$.
 
Eventually, we can get the result from  \eqref{end:eq:density-before-final} and \eqref{eq:monotonicity}:
$$ \pi(H_3^r)\ge\frac1{360r\log r}\pi(H_3^L)\ge\frac1{360r\log r}\,
 \pi(H_3^{\lceil129\log r\rceil}).$$
The proof is complete.  
\end{proof}

We now start to Prove Theorem \ref{end:thm:main} using Theorem \ref{end:thm:recursion}.  The following elementary estimate controls the product arising in the iteration.

\begin{lemma}\label{end:lem:product} 
Fix constants $C>1$ and $0<c<1$.  Choose $\alpha\geq e$ such that
$\lceil C\log x\rceil\le x/2$ for every $x\ge \alpha$.  Starting from $u_0=r$, define
$u_{j+1}=\lceil C\log u_j\rceil$ while $u_j\ge \alpha$, and let $\beta$ be the first index for which $u_\beta<\alpha$.  Then, for every fixed $\varepsilon>0$ and all sufficiently large $r$, we have
$$ c^\beta\prod_{j=0}^{\beta-1}\frac1{u_j\log u_j}
 \ge\frac1{r(\log r)^{2+\varepsilon}}.$$
\end{lemma}

\begin{proof}
For sufficiently large $r$, we have
$u_1\le C\log r+1\le(2C+2)\log r$.
For $u_2$, we have similar calculations, 
\begin{align*}
  u_2&\le C\log(u_1)+1\le (2C+2)\log u_1\le (2C+2)\log ((2C+2)\log r) \\
 &=(2C+2)\log \log r+(2C+2)\log (2C+2)\le (2C+2)^2\log \log r.
\end{align*}

By the definition of $\alpha$ and the index $\beta$, we know $u_{j+1}=\lceil C\log u_j\rceil \le u_j/2$ for any $1\le j<\beta $, and hence $u_{2+j}\le u_2/2^j$.
Let $h=\lfloor\log_2u_2\rfloor+1$.  If  $\beta >2+h$, then   $u_{2+h}\le \frac{u_2}{2^h}<1<\alpha$, contradicting the definition of $\beta$. Hence
\begin{equation}\label{end:eq:T-bound}
 \beta\le3+\log_2\big((2C+2)^2\log\log r\big)
 =O(\log\log\log r).
\end{equation}
In particular, $\beta=o(\log\log r)$.

We next estimate the product in the denominator.  Since $u_0=r$ and
$u_1\le(2C+2)\log r$, then
$$ (u_0\log u_0)(u_1\log u_1)
 \le(2C+2)r(\log r)^2\log((2C+2)\log r).$$
For every $2\le j<\beta$,  since $u_j\ge \alpha\ge e$, we have
$u_j\le u_2\le(2C+2)^2\log\log r$ and
$\log u_j\le u_j$.  It follows that
$u_j\log u_j\le u_j^2\le((2C+2)^2\log\log r)^2$.
Using \eqref{end:eq:T-bound}, we obtain
\begin{align*}
 &\log\prod_{j=2}^{\beta-1}(u_j\log u_j)
\le \log((2C+2)^{4\beta}(\log\log r)^{2\beta})\\
=&2\beta\log\big((2C+2)^2\log\log r\big)
 =O((\log\log\log r)^2)
 =o(\log\log r).   
\end{align*}
Also,
$\log\log((2C+2)\log r)=O(\log\log\log r)=o(\log\log r)$.
Combining the estimates for the first two factors and the remaining factors,
$$ \log\prod_{j=0}^{\beta-1}(u_j\log u_j)
 \le\log r+2\log\log r+o(\log\log r).$$
Therefore, for sufficiently large $r$, the last error term is at most
$(\varepsilon/2)\log\log r$, and hence
\begin{equation}\label{end:eq:prod-upper}
 \prod_{j=0}^{\beta-1}(u_j\log u_j)
 \le r(\log r)^{2+\varepsilon/2}.
\end{equation}

Finally, since $\beta=o(\log\log r)$ and $|\log c|$ is fixed, we have $\beta|\log c|\le(\varepsilon/2)\log\log r$ for sufficiently large $r$.  
Moreover, $c^\beta=\exp(\beta\log c)=\exp(-\beta|\log c|)\ge(\log r)^{-\varepsilon/2}$.
Together with \eqref{end:eq:prod-upper}, we obtain
$$
 \frac{c^\beta}{\displaystyle\prod_{j=0}^{\beta-1}(u_j\log u_j)}
 \ge\frac{1}{r(\log r)^{2+\varepsilon}}.
$$
The proof is complete.
\end{proof}

\begin{proof}[\bfseries{Proof of Theorem \ref{end:thm:main}}]
Let $c=1/360$ and $C=129$, and let $r_0$ be the constant from Theorem \ref{end:thm:recursion}.  
Increase $r_0$ so that $\lceil C\log x\rceil\le x/2$ and $\lceil C\log x\rceil\ge2$ for every $x\ge r_0$.  
Starting from $u_0=r$, define $u_{j+1}=\lceil C\log u_j\rceil$ while $u_j\ge r_0$, and 
let $\beta$ be the first index for which $u_\beta<r_0$.  
Applying Theorem \ref{end:thm:recursion} at $u_0,u_1,\ldots,u_{\beta-1}$, we have
\begin{equation}\label{end:eq:iterate}
 \pi(H_3^r)
 \ge\pi(H_3^{u_\beta})c^\beta\prod_{j=0}^{\beta-1}\frac1{u_j\log u_j}.
\end{equation}
The second condition on $r_0$ gives $2\le u_\beta<r_0$.  
By  (\ref{eq:monotonicity}) and  Theorem \ref{end:lem:positive}, we know
$\pi (H_3^{u_{\beta}})\ge \pi (H_3^{r_0})\ge r_0^{-2}$.  
Let $c_{\varepsilon}=r_0^{-2}$ and apply Lemma \ref{end:lem:product} with $\alpha=r_0$ to \eqref{end:eq:iterate}, we obtain, for every fixed $\varepsilon>0$,
$$ \pi(H_3^r)\ge\frac{c_{\varepsilon}}{r(\log r)^{2+\varepsilon}}$$
for all sufficiently large $r$.

Finally, we  show that $\pi(H_3^r)=r^{-1+o(1)}$.
For every fixed $\varepsilon>0$, the estimate above gives $\log\pi(H_3^r)\ge-\log r-(2+\varepsilon)\log\log r+\log c_\varepsilon$.  Together with $\pi(H_3^r)\le1/r$, this yields
$-1-(2+\varepsilon)\frac{\log\log r}{\log r}+\frac{\log c_\varepsilon}{\log r}\le\frac{\log\pi(H_3^r)}{\log r}\le-1$.  
Let $r\to\infty$. we conclude that $\log\pi(H_3^r)=-(1+o(1))\log r$, 
which is equivalent to $\pi(H_3^r)=r^{-1+o(1)}$.
\end{proof}

\section{Results for the graph \texorpdfstring{$H_k^r$}{Hkr}}\label{sec:3}

In this section,  our main focus is on studying the Tur\'an density $\pi(H_k^r)$. Our method still based on the circular construction by Sidorenko.  Before we start the proof of Theorem \ref{thm:main}, let us review Sidorenko's construction first.
 
Let $\mathcal C$ be the unit circle in the complex plane: $\mathcal C=\{z: |z|=1\}$. For any $r$  randomly chosen points $A=\{z_1,\cdots,z_r\}$ on this cycle, let  $\Delta(A)$ be the the length of the shortest arc in $\mathcal C$ that contains at least $k-1$ elements of $A$.
The $r$-set $A$ is an edge if and only if 
\[0\le arg(\prod_{i=1}^{r}z_i)\le \Delta(A).\]
In this infinite $r$-graph, Sidorenko proved it is $H_k^r$-free and the edge density is $\mathbb{E}(\xi_{r,k-1})=\mathbb{E}(\frac{\Delta(A)}{2\pi})$.

One may note, let $D_i$ be the length of the arc from $z_i$ to $z_{i+1}$. Then
\[\Delta(A)=\min_{1\le i\le r} \sum_{j=0}^{k-3}D_{i+j}.\]
All indices are taken modulo $r$. But since $D_1+\cdots +D_r=2\pi$, the variables $(D_1,\ldots, D_r)$ are not independent. Therefore, we will use the R\'enyi exponential representation to deal with this problem.

Let $X_1,X_2,\ldots$ be independent exponential random variables with density $e^{-x}$ for $x\ge0$ and $\mathbb{E}(X_i)=1$.  
Let $S_r=\sum_{j=1}^rX_j$.  In the cyclic sums below, every index outside $\{1,\ldots,r\}$ is read modulo $r$.
For $1\le\ell\le r-1$, we define
$$Y_i^{(\ell)}=\sum_{j=0}^{\ell-1}X_{i+j}
\quad \text{and} \quad
M_{r,\ell}=\min_{1\le i\le r}Y_i^{(\ell)}.$$

The following lemma is due to R\'enyi.
\begin{lemma}[R\'enyi~\cite{Renyi1953}]\label{lem:renyi1953}
The vector $(\frac{D_1}{2\pi },\ldots, \frac{D_r}{2\pi })$ have the same distribution as
$\left(\frac{X_1}{S_r},\ldots,\frac{X_r}{S_r}\right),$
and this normalized vector is independent of $S_r$.  Consequently, for $1\le m\le r-1$, we have
\begin{equation}\label{eq:spacing-expectation}
        \mathbb{E}(\xi_{r,m+1})=\frac1r\mathbb{E}(M_{r,m}).
\end{equation}
\end{lemma}

\medskip

Let $Z_j=X_j-1$, so the variables $Z_j$ are independent, have mean zero, and have variance one.  We first state a known maximal inequality for ordinary partial sums and then use a block argument for cyclic sums. The following special case of the maximal Bernstein inequality is standard. 
It applies to centered exponential variables. 
See Doob~\cite{Doob1953} and Kevei and Mason~\cite{KeveiMason2011}.

\begin{lemma}\label{lem:partial-max}
Let $Z_1,\ldots,Z_n$ be independent copies of $X_1-1$, and let $T_j=\sum_{i=1}^jZ_i$.  For every $u\ge0$, we have
\begin{equation}\label{eq:partial-tail}
 \mathbb{P}\left(\max_{1\le j\le n}|T_j|\ge u\right)
 \le 2\exp\left[-\frac14\min\left\{\frac{u^2}{n},u\right\}\right].
\end{equation}
\end{lemma}

We now pass from ordinary partial sums to cyclic sums.

\begin{lemma}\label{lem:scan}
Let $1\le\ell\le r/2$ and, with cyclic indices, define
$W_i^{(\ell)}=\sum_{j=0}^{\ell-1}Z_{i+j}$ for $1\le i\le r$.  Put
$L_\ell=\log(er/\ell)$.  Then
\begin{equation}\label{eq:scan-expectation}
 \mathbb{E}\left(\max_{1\le i\le r}|W_i^{(\ell)}|\right)
 \le 128\left(\sqrt{\ell L_\ell}+L_\ell\right).
\end{equation}
\end{lemma}

\begin{proof}
To calculate the maximum of
$W_i^{(\ell)}=\sum_{j=0}^{\ell-1}Z_{i+j}$, where $1\le i\le r$, we divide them into groups according to their starting indices.  More precisely, divide the starting indices $1,\ldots,r$ into consecutive blocks, each containing at most $\ell$ indices.  The number of blocks is
$\lceil r/\ell\rceil\le2r/\ell$, since $\ell\le r/2$.  Fix a block with starting indices
$a,a+1,\ldots,a+d-1$, where $d\le\ell$, and read all indices modulo $r$.  The cyclic sum beginning at $a+s$, where $0\le s\le d-1$, involves the variables
$Z_{a+s},\ldots,Z_{a+s+\ell-1}$.  Hence all cyclic sums whose starting indices lie in this block involve only the consecutive variables
$Z_a,\ldots,Z_{a+d+\ell-2}$.
Their number is $d+\ell-1\le2\ell-1<r$.  Thus, even if this interval passes through the index $r$, no variable occurs twice.  List these distinct variables in their cyclic order as $Z'_1,\ldots,Z'_n$, where
$n=d+\ell-1\le2\ell$, and define
$Q_0=0$ and $Q_j=\sum_{h=1}^jZ'_h$ for $1\le j\le n$ (see Figure \ref{fig:cyclic-block}).

\begin{figure}[htbp]
\centering
\begin{tikzpicture}[
    x=1.25cm,
    y=0.85cm,
    >=Latex,
    every node/.style={font=\small}
]

\node[anchor=east] at (-0.35,3.55) {$W_a^{(\ell)}$};
\draw[line width=1.1pt] (0,3.55)--(6,3.55);
\draw (0,3.38)--(0,3.72);
\draw (6,3.38)--(6,3.72);
\draw[decorate,decoration={brace,amplitude=4pt}]
      (0,4)--(6,4)
      node[midway,above=4pt] {$\ell$ consecutive variables};

\node[anchor=east] at (-0.35,2.85) {$W_{a+1}^{(\ell)}$};
\draw[line width=1.1pt] (1,2.85)--(7,2.85);
\draw (1,2.68)--(1,3.02);
\draw (7,2.68)--(7,3.02);

\node at (-0.35,2.15) {$\vdots$};
\node at (4.5,2.15) {$\vdots$};

\node[anchor=east] at (-0.35,1.45) {$W_{a+d-1}^{(\ell)}$};
\draw[line width=1.1pt] (3,1.45)--(9,1.45);
\draw (3,1.28)--(3,1.62);
\draw (9,1.28)--(9,1.62);

\draw[thin] (0,0.75)--(9,0.75);
\foreach \x in {0,1,3,6,7,9}
    \draw (\x,0.63)--(\x,0.87);

\node[font=\scriptsize] at (0,0.42) {$Z_a$};
\node[font=\scriptsize] at (1,0.42) {$Z_{a+1}$};
\node[font=\scriptsize] at (2,0.42) {$\cdots$};
\node[font=\scriptsize] at (3,0.42) {$Z_{a+d-1}$};
\node[font=\scriptsize] at (4.5,0.42) {$\cdots$};
\node[font=\scriptsize] at (6,0.42) {$Z_{a+\ell-1}$};
\node[font=\scriptsize] at (7,0.42) {$Z_{a+\ell}$};
\node[font=\scriptsize] at (8,0.42) {$\cdots$};
\node[font=\scriptsize] at (9,0.42) {$Z_{a+d+\ell-2}$};

\draw[decorate,decoration={brace,mirror,amplitude=5pt}]
      (0,0.05)--(9,0.05)
      node[midway,below=7pt]
      {$d+\ell-1$ distinct variables};

\node[align=center] at (4.5,-1.5)
{Relabel these variables in their cyclic order as
 $Z'_1,\ldots,Z'_n$,\\
 and define $Q_0=0$ and $Q_j=\sum_{h=1}^jZ'_h$.};

\end{tikzpicture}
\caption{one block of starting indices: $a,a+1,\ldots,a+d-1$}
\label{fig:cyclic-block}
\end{figure}
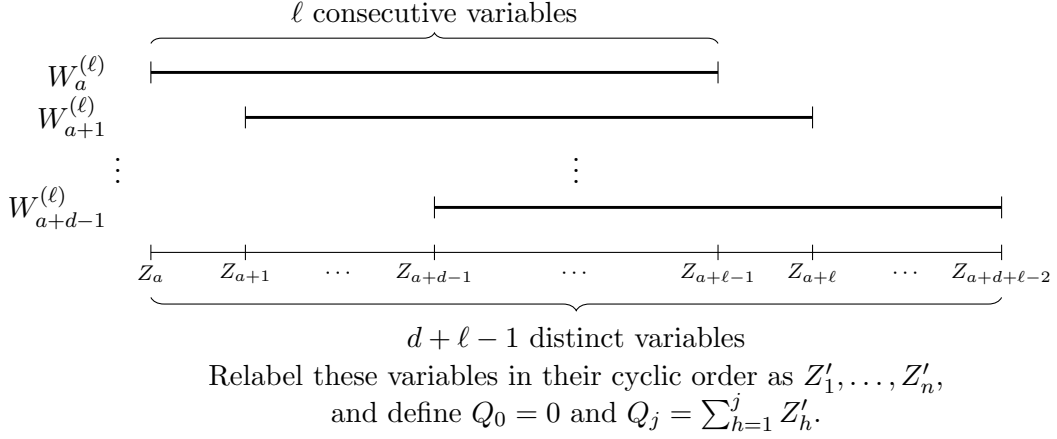

For each $0\le s\le d-1$, the cyclic sum beginning at $a+s$ is
$W_{a+s}^{(\ell)}=Z'_{s+1}+\cdots+Z'_{s+\ell}=Q_{s+\ell}-Q_s$.
Since $s+\ell\le d+\ell-1=n$, both  $Q_{s+\ell}, Q_s$ are among
$Q_0,\ldots,Q_n$ and 
$|W_{a+s}^{(\ell)}|\le |Q_{s+\ell}|+|Q_s|$. Hence
$$ \max_{i\text{ in the fixed block}}|W_i^{(\ell)}|
 \le2\max_{0\le j\le n}|Q_j|.$$
It follows that, if the left-hand side is at least $t$, then
$\max_{0\le j\le n}|Q_j|\ge t/2$.  
Applying Lemma \ref{lem:partial-max} with $u=t/2$ gives
\begin{align*}
   &\mathbb{P}\left(\max_{i\text{ in the fixed block}}|W_i^{(\ell)}|\ge t\right)\\
   \le &\mathbb{P}\left(\max_{0\le j\le n}|Q_j|
   \ge t/2\right)
 \le2\exp\left[-\frac14\min\left\{\frac{t^2}{4n},\frac t2\right\}\right]
 \le2\exp\left[-\frac1{32}\min\left\{\frac{t^2}{\ell},t\right\}\right]. 
\end{align*}
Indeed, since $n\le2\ell$, we have $t^2/(4n)\ge t^2/(8\ell)$, and therefore
$\frac14\min\{t^2/(4n),t/2\}\ge\frac1{32}\min\{t^2/\ell,t\}$.

Let $Y=\max_{1\le i\le r}|W_i^{(\ell)}|$.  There are at most $2r/\ell$ blocks, so the union bound gives
\begin{equation}\label{eq:scan-tail}
 \mathbb{P}(Y\ge t)\le \sum_{\text{all blocks}}\mathbb{P}\left(\max_{i\text{ in the fixed block}}|W_i^{(\ell)}|\ge t\right)
 \le\min\left\{1,\frac{4r}{\ell}
 \exp\left[-\frac1{32}\min\left\{\frac{t^2}{\ell},t\right\}\right]\right\}.
\end{equation}

We now calculate the expectation $\mathbb{E}(Y)$.  Since
$$
 \exp\left[-\frac1{32}\min\left\{\frac{t^2}{\ell},t\right\}\right]
 \le \exp\left(-\frac{t^2}{32\ell}\right)+\exp\left(-\frac t{32}\right)
$$
and $\min\{1,x+y\}\le\min\{1,2x\}+\min\{1,2y\}$ for all $x,y\ge0$, we have
$$ \mathbb{E}(Y)
 \le\int_0^\infty\min\left\{1,\frac{8r}{\ell}
       \exp\left(-\frac{t^2}{32\ell}\right)\right\}\,dt
   +\int_0^\infty\min\left\{1,\frac{8r}{\ell}
       \exp\left(-\frac t{32}\right)\right\}\,dt.$$
Let
$t_0=\sqrt{32\ell\log(8r/\ell)}$.
Then $(8r/\ell)\exp(-t_0^2/(32\ell))=1$.  Splitting the first integral at $t_0$ and using
$\int_x^\infty e^{-av^2}\,dv\le e^{-ax^2}/(2ax)$, we obtain
$$
 \int_0^\infty\min\left\{1,\frac{8r}{\ell}
       \exp\left(-\frac{t^2}{32\ell}\right)\right\}\,dt
 \le \sqrt{32\ell\log\frac{8r}{\ell}}
    +\frac{16\ell}{\sqrt{32\ell\log(8r/\ell)}}.
$$
Similarly, let $t_1=32\log(8r/\ell)$ and split the second integral at $t_1$ gives
$$
 \int_0^\infty\min\left\{1,\frac{8r}{\ell}
       \exp\left(-\frac t{32}\right)\right\}\,dt
 \le32\left(\log\frac{8r}{\ell}+1\right).
$$

Recall that $L_\ell=\log(er/\ell)$.  Since $\ell\le r/2$, we have
$8r/\ell\ge16$, $L_\ell\ge\log(2e)>1$, and
$\log(8r/\ell)\le2L_\ell$.  
It follows that
$$
 \sqrt{32\ell\log\frac{8r}{\ell}}
 \le8\sqrt{\ell L_\ell}
\quad \text{and} \quad
 \frac{16\ell}{\sqrt{32\ell\log(8r/\ell)}}
 \le2\sqrt{\ell}
 \le2\sqrt{\ell L_\ell}.
$$
Moreover,
$$
 32\left(\log\frac{8r}{\ell}+1\right)
 \le32(2L_\ell+1)
 \le96L_\ell.
$$
Combining these estimates, we obtain
$$ \mathbb{E}(Y)
 \le10\sqrt{\ell L_\ell}+96L_\ell
 \le128\left(\sqrt{\ell L_\ell}+L_\ell\right),$$
which proves \eqref{eq:scan-expectation}.
\end{proof}

\begin{proof}[\bfseries{Proof of Theorem \ref{thm:main}}]
Fix $3\le k\le r+1$ and let $m=k-2$.  Then $1\le m\le r-1$.
The known upper bound shows that
 $\pi(H_k^r)\le\frac{k-2}{r}=\frac mr.$
It remains to prove the corresponding lower bound by using the circular construction.

We first suppose $m\le r/2$.  Since
$Y_i^{(m)}=m+W_i^{(m)}$ for every $i$, taking the minimum over
$1\le i\le r$ gives
$M_{r,m}=m+\min_{1\le i\le r}W_i^{(m)}$.
Moreover,
$\min_{1\le i\le r}W_i^{(m)}
\ge-\max_{1\le i\le r}|W_i^{(m)}|$.
It follows from Lemma \ref{lem:scan} that
$$ \mathbb{E}(M_{r,m})
 \ge m-\mathbb{E}\left(\max_{1\le i\le r}|W_i^{(m)}|\right)
 \ge m-128\left(\sqrt{m\log\frac{er}{m}}+\log\frac{er}{m}\right).$$

We now suppose $m>r/2$ and let $\ell=r-m$.  Then $1\le\ell<r/2$.  Since
$m+\ell=r$, the terms in $Y_i^{(m)}$ and $Y_{i+m}^{(\ell)}$ together contain each of
$X_1,\ldots,X_r$ exactly once.  Therefore
$Y_i^{(m)}+Y_{i+m}^{(\ell)}=S_r$ for every $i$.  Since the map
$i\mapsto i+m$ is a permutation of the indices modulo $r$, we obtain
\begin{equation}\label{eq:complement-identity}
 M_{r,m}=S_r-\max_{1\le i\le r}Y_i^{(\ell)}.
\end{equation}
Since $Y_i^{(\ell)}=\ell+W_i^{(\ell)}$ and
$\mathbb{E}(S_r)=\sum_{j\le r} \mathbb{E}(X_j)=r$, we have
$$ \mathbb{E}(M_{r,m})
 =r-\ell-\mathbb{E}\left(\max_{1\le i\le r}W_i^{(\ell)}\right)
 \ge m-\mathbb{E}\left(\max_{1\le i\le r}|W_i^{(\ell)}|\right).$$
By Lemma \ref{lem:scan}, we obtain
$$ \mathbb{E}(M_{r,m})
 \ge m-128\left(\sqrt{\ell\log\frac{er}{\ell}}
                 +\log\frac{er}{\ell}\right).$$

Let $s=\min\{m,r-m\}$.  
In both cases, we have
$$
 \mathbb{E}(M_{r,m})
 \ge m-128\left(\sqrt{s\log\frac{er}{s}}
                 +\log\frac{er}{s}\right).
$$
By the R\'enyi representation (Lemma \ref{lem:renyi1953}) and Sidorenko's circular lower bound (Theorem \ref{thm:sidorenko}),
$$ \pi(H_k^r)
 \ge\mathbb{E}(\xi_{r,k-1})
 =\mathbb{E}(\xi_{r,m+1})
 =\frac1r\mathbb{E}(M_{r,m})
 \ge\frac{k-2}{r}-\frac{128}{r}
 \left(\sqrt{s\log\frac{er}{s}}+\log\frac{er}{s}\right).$$
This completes the proof of Theorem \ref{thm:main}.
\end{proof}

\begin{proof}[\bfseries{Proof of Corollary \ref{cor:applications}}]
Let $s=\min\{k-2,r-k+2\}$.
By Theorem \ref{thm:main},
\begin{equation}\label{eq:corollary-from-main}
0\le \frac{k-2}{r}-\pi(H_k^r)
\le \frac{128}{r}
\left(
\sqrt{s\log\frac{er}{s}}
+\log\frac{er}{s}
\right).
\end{equation}

Suppose first that $k/\log r\to c\in(0,\infty)$. 
Then $(k-2)/\log r\to c$ and $k=o(r)$, so $s=k-2$ for all sufficiently large $r$. 
Moreover,
$$\frac{\log(er/(k-2))}{\log r}=1+\frac{1-\log(k-2)}{\log r}\longrightarrow 1.$$
It follows that
$$\frac{\log(er/(k-2))}{k-2}\longrightarrow \frac1c.$$
Multiplying \eqref{eq:corollary-from-main} by $r/(k-2)$, we obtain
$$0\le1-\frac{r\pi(H_k^r)}{k-2}\le128\left(\sqrt{\frac{\log(er/(k-2))}{k-2}}+\frac{\log(er/(k-2))}{k-2}\right).$$
Consequently,
$$0\le 1-\frac{r\pi(H_k^r)}{k-2}\le 128\left(c^{-1/2}+c^{-1}\right)+o(1),$$
which proves part \textup{(i)}.

We next assume that $\log(er/k)=o(k)$. 
This condition implies $k\to\infty$, and hence $k-2\sim k$. 
If $k-2\le r/2$, then $s=k-2$ and
$$\log\frac{er}{k-2}=\log\frac{er}{k}+\log\frac{k}{k-2}=o(k-2).$$
Therefore,
$$\sqrt{(k-2)\log\frac{er}{k-2}}+\log\frac{er}{k-2}=o(k-2).$$
If $k-2>r/2$, then $s\le r$ and $s\ge1$, so
$$\sqrt{s\log\frac{er}{s}}+\log\frac{er}{s}
\le \sqrt{s\log(er)}+\log(er)
\le \sqrt{r\log(er)}+\log(er)=o(r).$$
Since $r<2(k-2)$, the last estimate is also $o(k-2)$.
Thus, in both cases, the right-hand side of
\eqref{eq:corollary-from-main} is $o((k-2)/r)$. Together with the
left-hand inequality in \eqref{eq:corollary-from-main}, this proves
part \textup{(ii)}.

Finally, suppose that $k/r\to\alpha\in(0,1)$. 
Then
$$\frac{k-2}{r}\longrightarrow\alpha,
\quad
\frac{r-k+2}{r}\longrightarrow1-\alpha,$$
and hence $s/r\to\beta$, where $\beta=\min\{\alpha,1-\alpha\}$. 
It follows that
$$\log\frac{er}{s}\longrightarrow\log\frac e\beta.$$
Consequently, we get 
$$
\frac{128}{r} \left(\sqrt{s\log\frac{er}{s}}+\log\frac{er}{s}\right)
=\frac{128}{\sqrt r}\sqrt{\frac{s}{r}\log\frac{er}{s}}+O\left(\frac1r\right)
=\frac{128\sqrt{\beta\log(e/\beta)}+o(1)}{\sqrt r},$$
as desired.
\end{proof}

\section*{Acknowledgement}
\noindent The author Yang is supported by National Key R\&D Program of China under grant number 2024YFA1013900, 
NSFC under grant number 12471327, the China Postdoctoral Science Foundation under grant number 2026M793375.
The author Zhu is supported by NSFC under grant number 12401454, Basic Research Program of Jiangsu Province(BK20241361).

\section*{Declaration}
During the preparation of this work, the authors used AI systems for various tasks, including generation of candidate
proofs. All AI
suggestions that were incorporated into the final version were verified and streamlined by the authors, who also tried to find closest matches in the literature to the key AI-generated mathematical ideas used in this paper. The authors take full responsibility for the correctness of the paper.

\end{document}